\documentclass[12pt]{article}
\usepackage{amsmath, amsthm, amssymb}
\usepackage{geometry}
\usepackage{hyperref}
\usepackage{titlesec}
\usepackage{tikz-cd}

\definecolor{cZongpu}{rgb}{0.1,0.45,0.03}
\definecolor{cChristian}{rgb}{0.6,0.5,0.0}

\title{MULTIGRADED BETTI NUMBERS OF VERONESE EMBEDDINGS}
\author{Christian Haase and Zongpu Zhang}

\date{January 31, 2026}

\newtheorem{theorem}{Theorem}[section]
\newtheorem{lemma}[theorem]{Lemma}
\newtheorem{corollary}[theorem]{Corollary}

\theoremstyle{definition}
\newtheorem{definition}[theorem]{Definition}
\theoremstyle{remark}
\newtheorem{remark}[theorem]{Remark}
\newtheorem{example}[theorem]{Example}
\newtheorem*{claim*}{Claim}
\newtheorem{claim}{Claim}
\newcommand{\defeq}{\mathrel{\mathop:}=}

\begin{document}

\maketitle

\begin{abstract}
In this paper, we study the multigraded Betti numbers of Veronese embeddings of projective spaces. Due to Hochster's formula, we interpret these multigraded Betti numbers in terms of the homology of certain simplicial complexes. By analyzing these simplicial complexes and applying Forman's discrete Morse theory, we derive vanishing and non-vanishing results for these multigraded Betti numbers.
\end{abstract}
\section{INTRODUCTION}
 A central open question in the study of syzygies is to determine the Betti table of $\mathbb{P}^m$ under the $d$-uple Veronese embedding. Veronese syzygies have been the focus of a great deal of attention, e.g.~\cite{zbMATH07330979, zbMATH07124911, zbMATH06609392, zbMATH06117694, zbMATH06620296, arXiv:2311.03625, zbMATH01587223, zbMATH05268889}. The case $m=1$ is well understood, but even the case $m=2$ is wide open.
 For the $m=2$ case, we do know which coarsely graded Betti numbers vanish~\cite{Aprodu,Green,Birkenhake}:
 \begin{align}
 \beta_{p,0}\neq0 & \iff p=0 \notag \\
 \beta_{p,1}\neq0 & \iff 1 \le p \le \binom{d+1}{2} \label{eq:coarse} \\
     \beta_{p,2}\neq0 & \iff 3d-2 \le p \le \binom{d+2}{2} -3 \,. \notag 
 \end{align}
 Here, we focus on the multigraded Betti numbers.
 In a major computational effort,~\cite{zbMATH07330979} recently computed all multigraded Betti numbers up to the sixth Veronese (and some up to the eighth). Our goal in this paper is to better understand the multigraded Betti numbers of Veronese embeddings of projective spaces. In particular, we prove new vanishing and non-vanishing results for these multigraded Betti numbers. Related vanishing theorems were proved with other methods by Castryck, Lemmens, and Hering~\cite{CastryckLemmensHering}. Their bounds are incomparable to ours.

What are the multigraded Betti numbers of Veronese embeddings of projective spaces? Let $k$ be any field. Fix integers $d\geq2$ and $m\geq 2$. We consider the set $\mathcal{A}=\{\mathbf{a}\in \mathbb{N}^{m+1}\mid |\mathbf{a}|=d\}$, where $|\mathbf{a}|\defeq\mathbf{a}_0+\mathbf{a}_1+\cdots+\mathbf{a}_m$ denotes the sum of all coordinates of $\mathbf{a}$. Let $\mathbf{a}^i$ be the $i$-th element of $\mathcal{A}$ with respect to the lexicographical order, let $n\defeq \binom{d+m}{m}$ be the cardinality of $\mathcal{A}$, and denote $\mathbb{N}\mathcal{A}\subset\mathbb{N}^{m+1}$ be the semigroup generated by $\mathcal{A}$.

Consider the semigroup homomorphism
\[\pi\colon\mathbb{N}^n\to \mathbb{Z}^{m+1}, \quad\mathbf{u}=(u_1,...,u_n)\mapsto u_1\textbf{a}^1+\cdots+u_n\mathbf{a}^n.\]
Let $S=k[x_1,\dots,x_n]$. The homomorphism $\pi$ defines a multigrading on $S$, where $\deg(x_i)=\mathbf{a}^i$. And this induces a $k$-algebra homomorphism \(\hat{\pi}\colon S\to k[\mathbb{N}\mathcal{A}]\). Note that $k[\mathbb{N}\mathcal{A}]$ is the $d^{\text{th}}$ Veronese subring of $k[y_0,y_1,\dots,y_m]$, and it has an $S$-module structure via the map $\hat{\pi}$.
The $S$-module $k[\mathbb{N}\mathcal{A}]$ is multigraded, with graded components given by \[k[\mathbb{N}\mathcal{A}]_\mathbf{a}=\begin{cases}
    k\{\mathbf{y}^{\mathbf{a}}\}, &\text{if }  \mathbf{a}\in\mathbb{N}\mathcal{A},\\
    0,&\text{if }\mathbf{a}\notin \mathbb{N}\mathcal{A}.
\end{cases}\]
Our goal is to study the multigraded Betti numbers $\beta_{p,\mathbf{b}}=\beta_{p,\mathbf{b}}(k[\mathbb{N}\mathcal{A}])$ of the multigraded $S$-module $k[\mathbb{N}\mathcal{A}]$. Geometrically, this corresponds to studying the multigraded Betti numbers of the coordinate ring of $\mathbb{P}^m$ under the $d$-uple Veronese embedding $\mathbb{P}^m\to \mathbb{P}^{n-1}$. 

In general, it is not easy to determine the multigraded Betti numbers by direct methods, using just the definition. Instead, we will approach this problem using combinatorial methods. By Hochster's formula below, we interpret the multigraded Betti number $\beta_{p,\mathbf{b}}$ in terms of the reduced simplicial homology of a certain simplicial complex $\Delta_\mathbf{b}$ (cf. Definition~\ref{defDelta}).

Let $\mathbf{b} \in \mathbb{N}\mathcal{A}$ with $|\mathbf{b}|=dj$ for a positive integer $j$.
We give lower and upper bounds for the components of $\mathbf{b}$ beyond which $\Delta_{\mathbf{b}}$ is a cone and consequently the corresponding Betti numbers vanish.

\begin{theorem}\label{vanishing1}
   Let $j$ be a positive integer, and let $|\mathbf{b}|=dj$. If $\mathbf{b}_0\geq A_j$, then $\beta_{p,\mathbf{b}}=0$ for all $p\in\mathbb{Z}$.
\end{theorem}

Here, $A_j = A_j(m,d)$ is a combinatorial bound (cf.~\S\ref{VANISHING THEOREM 1}).

\begin{theorem}\label{vanishing2}
   Let $j$ be a positive integer such that $j\geq d+1$, and let $|\mathbf{b}|=dj$. If $\mathbf{b}_0\leq \tilde{l}_j$, then $\beta_{p,\mathbf{b}}=0$ for all $p\in \mathbb{Z}$.
\end{theorem}

Again, $\tilde{l}_j=\tilde{l}_j(m,d)$ is a combinatorial bound, but we do not have closed formulae (cf.~\S\ref{VANISHING THEOREM 2}). 

In other words, by symmetry, for each $s=0,1,\dots,m$, the set \[\operatorname{pr}_s(\{\mathbf{c}\in\mathbb{Z}^{m+1}\mid |\mathbf{c}|=dj, \text{ and } \beta_{p,\mathbf{c}}\neq 0 \text{ for some } p\in\mathbb{Z}\})\] has the upper bound $A_j-1$. We show that the upper bounds are optimal (cf. Theorem~\ref{optimalA}). It has the lower bound $\tilde{l}_j+1$ if we assume $j\geq d+1$. Here, $\operatorname{pr}_s$ denotes the projection onto the $s+1$-th coordinate.


Applying Forman's discrete Morse theory~\cite{zbMATH02041483}, we compute the exact Betti numbers when $\mathbf{b}_0=A_{p+1}-1$ and state the result as follows.
\begin{theorem}\label{nonvanishing D}
    
    Let $m\leq p\leq \binom{d+m-1}{m}-1$. Suppose $|\mathbf{b}|=d(p+1)$ and $\mathbf{b}_0=A_{p+1}-1$. We define a set $D$ that depends on $p$ and $\mathbf{b}$. Then the simplicial complex $\Delta_{\mathbf{b}}$ is homotopy equivalent to the wedge sum of $\#D$ copies of the sphere $S^{p-1}$. Hence, the Betti number $\beta_{p,\mathbf{b}}=\# D$.
\end{theorem}

 \begin{figure}[htbp]\label{impo}
    \centering
    \includegraphics[width=0.6\textwidth]{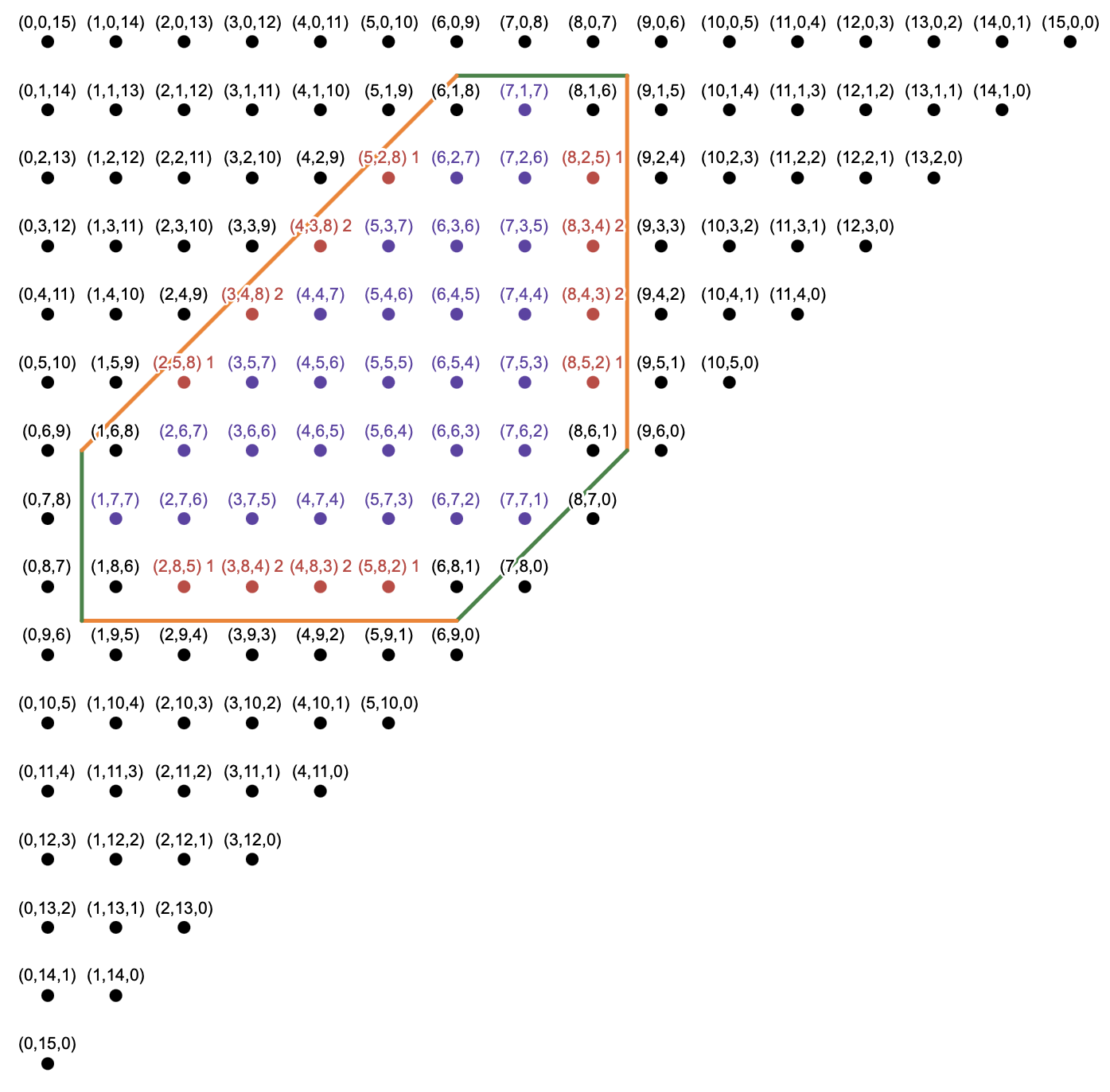} 
    \caption{Multigraded Betti numbers $\beta_{4,\mathbf{b}}$ for $|\mathbf{b}|=15$. }
    \label{fig:example}
\end{figure}

All these theorems can be illustrated by the example shown in Figure $1$. In this example, $m=2$, $d=3$, $p=4$, and $j=p+1=5$. We have $A_5=9$ and $\tilde{l}_5=0$. The orange lines represent the bounds $A_5$, and the green lines represent the bounds $\tilde{l}_5$. The black dots denote vanishing Betti numbers. For the red dots, the exact Betti numbers are computed and shown next to their coordinates. For the purple dots, the Betti numbers are generally unknown.

The simplicial complex $\Delta_{\mathbf{b}}$ has been studied under the name bounded degree (hyper)graph complex,
see for example the survey by Wachs~\cite{zbMATH05040242}.
The majority of those results focuses on the graph case $d=2$ with varying $m$, and thus is somewhat complementary to our work (varying $d$, sometimes $m=2$). For $m=d=2$ all multigraded Betti numbers are known.



For $d=2$, the quadratic Veronese, Reiner and Roberts~\cite{zbMATH01491052} studied the representation theoretic decomposition of the homology (in characteristic $0$) of $\Delta_\mathbf{b}$. As a corollary, one obtains a formula for $\dim_k \tilde{H}_{p-1}(\Delta_\mathbf{b};k)$ in terms of Kostka numbers. By Hochster's formula, this yields a formula for the multigraded Betti number $\beta_{p,\mathbf{b}}$ of the quadratic Veronese embedding.

Still for $d=2$, Dong~\cite{zbMATH01917896} studied the homotopy type of bounded degree graph complexes and provided a complete characterization of when $\Delta_\mathbf{b}$ is contractible~\cite[Theorem~3.8]{zbMATH01917896}.
For the special case of the quadratic Veronese, this contains our vanishing results.
He also showed that under certain conditions $\Delta_\mathbf{b}$ is homotopy equivalent to a wedge of equidimensional  spheres~\cite[Theorem~3.10]{zbMATH01917896}. 
The $d=2$ case of our Theorem~\ref{nonvanishing D} follows from Dong's considerations.

For general $d$, Bj\"orner, Lov\'asz, Vre\'cica and \v{Z}ivaljevi\'c~\cite{zbMATH00540563} and
Athanasiadis~\cite[Theorem~1.2]{zbMATH02123194}
studied the connectivity of the ``hypergraph matching complex'' $\Delta_{(1,\ldots,1)}$.
They show that the complex is $\lfloor \frac{m-2d}{d+1} \rfloor$-connected. The weight space technique in~\cite{zbMATH01643647} can be used to lift this result to bounded degree hypergraph complexes: $\Delta_\mathbf{b}$ is $\lfloor \frac{|\mathbf{b}|-2d-1}{d+1} \rfloor$-acyclic over a field of characteristic $0$. In the language of Betti numbers, this means that for each $1\leq q\leq m$, $\beta_{p,q}=0$ if $0\leq p\leq (q-1)d$. For the projective plane ($m=2$) this is a consequence of \eqref{eq:coarse}.
\medskip

The paper is organized as follows. After reviewing Hochster’s formula and basic background on discrete Morse theory in Section \ref{section2}, we prove Theorem \ref{vanishing1} in Section \ref{VANISHING THEOREM 1}, Theorem \ref{vanishing2} in Section \ref{VANISHING THEOREM 2}, and Theorem \ref{nonvanishing D} in Section \ref{NON-VANISHING RESULTS}. For simplicity, we assume $m=2$ in Section \ref{VANISHING THEOREM 1}, \ref{VANISHING THEOREM 2}, \ref{NON-VANISHING RESULTS}. The same arguments apply to $m>2$, and the general results are stated in Section \ref{general}. Section \ref{optimal} is devoted to showing that the bounds $A_{p+1}$ are optimal.

\subsection*{Acknowledgements}
We thank Milena Hering and the authors of~\cite{zbMATH07330979} for enlightening conversations and encouragement.
Both authors had the opportunity to discuss the present ideas during the Thematic Program in Commutative Algebra and Applications at the Fields Institute in Toronto, Canada.
We also thank Vic Reiner for pointing us to bounded degree (hyper)graph complexes and the like.

CH was supported by the SPP 2458 ``Combinatorial Synergies'', funded by the Deutsche Forschungsgemeinschaft (DFG, German Research Foundation). ZZ was supported by the Berlin Mathematical School of MATH+, the Berlin Mathematics Research Center.

    
\section{BACKGROUND}\label{section2}
\subsection{Hochster's formula}
In this subsection, we provide background on the simplicial complex $\Delta_\mathbf{b}$ and Hochster's formula. The references for these are \cite[Section~9.1]{zbMATH02190625} and \cite[Section~12.C]{zbMATH00835749}.

\begin{definition}\label{defDelta}
    For any degree $\mathbf{b}\in \mathbb{N}\mathcal{A}$, we define the simplicial complex $\Delta_\mathbf{b}$ by \[\Delta_\mathbf{b}=\left\{I\subset\{1,...,n\}\mid \mathbf{b}-\sum_{i\in I}\mathbf{a}^i \in \mathbb{N}\mathcal{A}\right\}.\]
\end{definition}
\begin{remark}
    Since $\mathcal{A}=\{\mathbf{a}\in \mathbb{N}^{m+1}\mid |\mathbf{a}|=d\}$, the condition $\mathbf{b}-\displaystyle\sum_{i\in I}\mathbf{a}^i \in \mathbb{N}\mathcal{A}$ is equivalent to the condition $\mathbf{b}-\displaystyle\sum_{i\in I}\mathbf{a}^i\geq \mathbf{0}$.
\end{remark}
\begin{theorem}[Hochster's formula]
    $\beta_{p,\mathbf{b}}=\dim_k\tilde{H}_{p-1}(\Delta_\mathbf{b};k)$.
\end{theorem}
\subsection{Discrete Morse theory}
In this subsection, we provide background on discrete Morse theory (cf.~\cite{zbMATH02041483, zbMATH07258477}).

Many mathematical problems naturally lead to the study of the topology of simplicial complexes. However, general techniques for this purpose are limited. On the other hand, powerful and useful theories—such as Morse theory—have been developed for smooth manifolds. Forman's discrete Morse theory, which may be applied to any simplicial complex, is a combinatorial adaptation of classical Morse theory.

We begin with a simplicial complex $\Delta$. If $\alpha\in\Delta$ and $\dim\alpha=q$, then we sometimes denote this by $\alpha^{(q)}$. For simplices $\alpha$ and $\beta$ we will use the notation $\alpha<\beta$ to indicate that $\alpha$ is a proper face of $\beta$. We now introduce the discrete analogue of vector fields.
\begin{definition}[{\cite[Definition~3.3]{zbMATH02041483}}]
A discrete vector field $V$ on $\Delta$ is a collection of pairs $\{\alpha^{(q)}<\beta^{(q+1)}\}$ of simplices of $\Delta$ such that each simplex is in at most one pair of $V$.
\end{definition}
Once a discrete vector field $V$ on $\Delta$ is given, a $V$-path is a sequence of simplices 
\[\alpha_0^{(q)},\beta_0^{(q+1)},\alpha_1^{(q)},\beta_1^{(q+1)},  \alpha_2^{(q)},\dots,\beta_h^{(q+1)},\alpha_{h+1}^{(q)}\]
such that for each $i=0,\dots,h$, $\{\alpha_i<\beta_i\}\in V$ and $\beta_i>\alpha_{i+1}\neq \alpha_i$. We say such a path is a non-trivial closed path if $h\geq 1$ and $\alpha_0=\alpha_{h+1}$.

A discrete Morse function on $\Delta$~\cite[Definition~2.1]{zbMATH02041483} is a function which, roughly speaking, assigns higher numbers to higher-dimensional simplices, with at most one exception, locally, at each simplex. Once a discrete Morse function $f$ is given, the discrete gradient vector field of $f$ is defined as a collection $W$ of pairs $\{\alpha^{(q)}<\beta^{(q+1)}\}$ of simplices, where $\{\alpha^{(q)}<\beta^{(q+1)}\}$ is in $W$ if and only if $f(\beta)\leq f(\alpha)$. A simplex is critical~\cite[Definition~2.3]{zbMATH02041483} if and only if it is not in $W$. 
\begin{theorem}[{\cite[Theorem~3.5]{zbMATH02041483}}]\label{non-trivial closed}
    A discrete vector field $V$ is the gradient vector field of a discrete Morse function if and only if there are no non-trivial closed $V$-paths.
\end{theorem}

\begin{example}\label{dvf}
    Let $\Delta$ be a simplicial complex, and let $v$ be a vertex of $\Delta$. We will construct a gradient vector field of a discrete Morse function. Consider the following collection $V$ of pairs \(
\{\, \sigma < \sigma \cup \{v\} \,\},
\)
where $\sigma$ is a nonempty simplex, $v \notin \sigma$, and $\sigma \cup \{v\}$ is a simplex.  

To verify that $V$ is a discrete vector field on $\Delta$, it suffices to show that every nonempty simplex $\eta\in\Delta$ is in at most one pair of $V$. 

\textbf{Case 1:} $v \in \eta$  

If $v \in \eta$, then
\[
\begin{cases}
\eta = \{v\}, & \text{then } \eta \text{ is not in any pair of } V, \\
\{v\}<\eta, & \text{then } \eta \text{ is in the pair } \{\eta\setminus\{v\}<\eta\}.
\end{cases}
\]
\textbf{Case 2:} $v \notin \eta$  

If $v \notin \eta$, then
\[
\begin{cases}
\eta\cup\{v\} \text{ is a simplex }, & \text{then } \eta \text{ is in the pair } \{\eta<\eta\cup\{v\}\}, \\
\eta\cup\{v\} \text{ is not a simplex }, & \text{then } \eta \text{ is not in any pair of } V.
\end{cases}
\]

Suppose there is a non-trivial closed $V$-path 
\[\sigma_0, \sigma_0\cup\{v\},\sigma_1, \sigma_1\cup\{v\},\cdots,\sigma_h, \sigma_h\cup\{v\},\sigma_{h+1}.\]
Since $\sigma_1$ is not equal to $\sigma_0$, we must have $v\in\sigma_1$, which contradicts the fact that $\{\sigma_1, \sigma_1\cup\{v\}\}\in V$.

Therefore, there is no non-trivial closed $V$-path. By Theorem~\ref{non-trivial closed}, the collection $V$ is the gradient vector field of a discrete Morse function. The critical simplices are $\{v\}$ and $\{\sigma\in\Delta_\mathbf{b}\mid \sigma\cup \{v\}\notin \Delta_\mathbf{b}\}$ (the complement of the closed star of $\{v\}$).
\end{example}
We can now state the main theorem of discrete Morse theory.
\begin{theorem}[{\cite[Theorem~2.5]{zbMATH02041483}}]
    Suppose $\Delta$ is a simplicial complex with a discrete Morse function. Then $\Delta$ is homotopy equivalent to a CW complex with exactly one cell of dimension $q$ for each critical simplex of dimension $q$.
\end{theorem}
\begin{corollary}
    Let $\Delta$ be a simplicial complex, and let $v$ be a vertex of $\Delta$. We define $N_{v,q}$ to be the number of $q$-cells in $\{\sigma\in\Delta\mid \sigma\cup \{v\}\notin \Delta\}$. Then $\Delta$ is homotopy equivalent to a CW complex with exactly $(N_{v,0}+1)$ $0$-cells, and $N_{v,q}$ $q$-cells for $q\geq 1$.
\end{corollary}

\begin{theorem}[{\cite[Theorem~2.11]{zbMATH02041483}}]
   Let $\Delta$ be a simplicial complex with a discrete Morse function. Let $m_q$ denote the number of critical simplices of dimension $q$. Then
    \[m_q\geq \dim \tilde{H}_q(\Delta;k).\]
\end{theorem}
By Hochster's formula, we obtain the following upper bounds for multigraded Betti numbers.
\begin{corollary}\label{Nvq}
   Let $\mathbf{b}\in\mathbb{N}\mathcal{A}$. Let $v$ be a vertex of $\Delta_\mathbf{b}$. Recall that $N_{v,q}$ denotes the number of $q$-cells in $\{\sigma\in\Delta_\mathbf{b}\mid \sigma\cup \{v\}\notin \Delta_\mathbf{b}\}$. We define $N_q\defeq \displaystyle\min_{v\in \operatorname{Vert}(\Delta_\mathbf{b})}N_{v,q}$. Then 
    \[N_{q}\geq \beta_{q+1,\mathbf{b}}.\]
\end{corollary}

\section{VANISHING THEOREM 1}\label{VANISHING THEOREM 1}
The goal of this section is to prove Theorem~\ref{vanishing1}.
\begin{theorem}\label{A_j}
    Let $j$ be a positive integer, and let $|\mathbf{b}|=dj$. If $\mathbf{b}_0\geq A_j\defeq \displaystyle\sum_{i\in\{1,...,j\}\cap \{1,...,n\}}\mathbf{a}^i_0$, then $\Delta_\mathbf{b}$ is a cone over $(d,0,0)$. 
\end{theorem}

\begin{proof}
Suppose $\sigma\in\Delta_\mathbf{b}$, then we have  $\mathbf{b}-\displaystyle\sum_{i\in \sigma}\mathbf{a}^i\in\mathbb{N}^3$. We want to show $\sigma\cup \{1\}\in \Delta_\mathbf{b}$. If $1\in \sigma$, we are done. Now we assume $1\notin \sigma$. It suffices to show \begin{equation*}
    \mathbf{b}_0-d-\displaystyle\sum_{i\in \sigma}\mathbf{a}^i_0\geq 0.\tag{*}
\end{equation*}
We distinguish two cases. 
\begin{itemize}
    \item  $1\leq j\leq n-1$. 
    
We first show that $|\sigma|\leq j-1$. 

We prove this by contradiction. Assume this is not the case, i.e. $|\sigma|\geq j$. Note that $\mathbf{b}-\displaystyle\sum_{i\in \sigma}\mathbf{a}^i\in\mathbb{N}^3$ implies $|\sigma|\leq j$. This means $|\sigma|=j$, and then we have 
\[\mathbf{b}_0\geq \sum_{i=1}^j\mathbf{a}^i_0>\sum_{i=2}^{j+1}\mathbf{a}^i_0\geq \sum_{i\in\sigma}\mathbf{a}^i_0.\] Hence $dj-\mathbf{b}_0<dj-\displaystyle\sum_{i\in\sigma}\mathbf{a}^i_0$. Since $|\sigma|=j$, we must have $dj=|\displaystyle\sum_{i\in\sigma}\mathbf{a}^i|$. Thus, we obtain $\mathbf{b}_1+\mathbf{b}_2<\displaystyle\sum_{i\in\sigma}\mathbf{a}^i_1+\displaystyle\sum_{i\in\sigma}\mathbf{a}^i_2$. In addition, $\mathbf{b}-\displaystyle\sum_{i\in \sigma}\mathbf{a}^i\in\mathbb{N}^3$ implies $\mathbf{b}_1\geq \sum_{i\in \sigma}\mathbf{a}^i_1$ and $\mathbf{b}_2\geq \sum_{i\in \sigma}\mathbf{a}^i_2$, then $\mathbf{b}_1+\mathbf{b}_2\geq\displaystyle\sum_{i\in\sigma}\mathbf{a}^i_1+\displaystyle\sum_{i\in\sigma}\mathbf{a}^i_2$, which is a contradiction. 

Since $|\sigma|\leq j-1$, we know $\displaystyle\sum_{i=2}^j\mathbf{a}^i_0\geq \displaystyle\sum_{i\in\sigma}\mathbf{a}^i_0$.
By assumption, $\mathbf{b}_0\geq \displaystyle\sum_{i=1}^j\mathbf{a}^i_0$. Thus, we have 
\[\mathbf{b}_0-d\geq  \displaystyle\sum_{i=1}^j\mathbf{a}^i_0-d= \displaystyle\sum_{i=2}^j\mathbf{a}^i_0\geq \displaystyle\sum_{i\in\sigma}\mathbf{a}^i_0.\]
Therefore, \text{(*)} is true.

\item  $j\geq n$. 

By assumption, $\mathbf{b}_0\geq \displaystyle\sum_{i=1}^n\mathbf{a}^i_0$. Then we have 
\[\mathbf{b}_0-d\geq \sum_{i=2}^n\mathbf{a}^i_0\geq \sum_{i\in\sigma}\mathbf{a}^i_0.\]

\end{itemize}
\end{proof}

We can explicitly write the expression for $A_j$ as follows:
\[A_j\defeq\displaystyle\sum_{i\in\{1,...,j\}\cap \{1,...,n\}}\mathbf{a}^i_0=\begin{cases}
        (d-k)j+\binom{k+2}{3} & \text{if}\quad j\in [\binom{k+1}{2},\binom{k+2}{2}]\quad \text{for some}\quad 1\leq k \leq d,\\
    \binom{d+2}{3} &\text{if}\quad j\geq n.
 
    \end{cases}\]

\section{VANISHING THEOREM 2}\label{VANISHING THEOREM 2}
The goal of this section is to prove Theorem~\ref{vanishing2}.
\begin{lemma}\label{monotonicity}
  Let $A$ be the integer $\displaystyle\sum_{i=1}^n\mathbf{a}^i_0$. For each $0\leq l \leq A$, define \[C_l \defeq\left\{ \sigma \subset \{1, \dots, n\} \;\middle|\; \displaystyle\sum_{i \in \sigma} \mathbf{a}_0^i = l \right\}\ \text{and}\ f_l\defeq\displaystyle\max_{\sigma \in C_l} \sum_{i \in \sigma} \mathbf{a}_1^i.\] Then $C_l\neq \emptyset$, and $f_{l+1}\geq f_l$ for each $0\leq l\leq A-1$. 
\end{lemma}
\begin{proof}
\textbf{Base case}: The set $C_0$ is nonempty because $\{n\}$ is in $C_0$.

\textbf{Inductive hypothesis}: Assume $C_l$ is nonempty. We aim to show that $C_{l+1}$ is nonempty.

To begin, choose an element $\sigma\in C_l$ such that $\displaystyle\sum_{i\in\sigma}\mathbf{a}^i_1=f_l$.

\begin{claim*}
    $\{(1,0,d-1),(2,0,d-2),\ldots,(d,0,0)\}$ is not a subset of $\sigma$.
\end{claim*}
\begin{proof}[Proof of Claim]
  We prove this by contradiction, assuming that it is a subset of $\sigma$. 
  
  Define $I\defeq\{(\alpha,\beta,\gamma) \in \mathbb{N}^3\mid \alpha\neq 0, \beta\neq 0, \alpha+\beta+\gamma=d\}$. If there is an $(\alpha,\beta,\gamma)$ in $I$ such that $(\alpha,\beta,\gamma)\notin \sigma$, then \[\displaystyle\sum_{i\in (\sigma\setminus \{(\alpha,0,d-\alpha)\})\cup \{(\alpha,\beta,\gamma)\}}\mathbf{a}^i_1>\sum_{i\in\sigma}\mathbf{a}^i_1=f_l,\] where $(\sigma\setminus \{(\alpha,0,d-\alpha)\})\cup \{(\alpha,\beta,\gamma)\}$ is in $C_l$. This is impossible since $f_l$ is the maximum. Hence $I\subset \sigma$.

  Then \[A-1\geq l=\sum_{i\in\sigma}\mathbf{a}^i_0= \sum_{i\in I}\mathbf{a}^i_0+ 1+2+\cdots+d=A,\] which is a contradiction. We must have $\{(1,0,d-1),(2,0,d-2),\ldots,(d,0,0)\}\not\subset \sigma$. 
\end{proof}

    Let $(k,0,d-k)$ be the first element in $\{(1,0,d-1),(2,0,d-2),\ldots,(d,0,0)\}$ such that $(k,0,d-k)\notin \sigma$. Then $(\sigma\setminus \{(k-1,0,d-k+1)\})\cup \{(k,0,d-k)\}\in C_{l+1}$. Hence $C_{l+1}\neq \emptyset$. And
     \[f_{l+1}\geq \sum_{i\in (\sigma\setminus \{(k-1,0,d-k+1)\})\cup \{(k,0,d-k)\}}\mathbf{a}^i_1=\sum_{i\in \sigma}\mathbf{a}^i_1=f_l.\]
\end{proof}

\begin{remark}
    By symmetry, we can also define $f_l$ using the third coordinate, that is, $f_l\defeq\displaystyle\max_{\sigma \in C_l} \sum_{i \in \sigma} \mathbf{a}_2^i$.
\end{remark}
 Assume $j\geq d+1$. By Lemma \ref{monotonicity}, we know that $\phi(l)\defeq \lceil \frac{dj-l}{2} \rceil-f_l$ is a decreasing function on $\{0,...,A_j-1\}$. Let $\tilde{l}_j$ be the largest integer $l$ in $\{0,...,A_j-1\}$ such that $\phi(l)\geq 0$. 
 \begin{theorem}\label{l}
Let $j$ be a positive integer such that $j\geq d+1$, and let $|\mathbf{b}|=dj$. If $\mathbf{b}_0\leq \tilde{l}_j$, then $\Delta_\mathbf{b}$ is a cone over $(0,d,0)$ or $(0,0,d)$. 
 \end{theorem}
\begin{proof}
   If $\mathbf{b}_0\leq \tilde{l}_j$, then \[\lceil \frac{\mathbf{b}_1+\mathbf{b}_2}{2}\rceil=\lceil \frac{dj-\mathbf{b}_0}{2}\rceil\geq\lceil \frac{dj-\tilde{l}_j}{2}\rceil.\] Note either $\mathbf{b}_1$ or $\mathbf{b}_2\geq \lceil \frac{\mathbf{b}_1+\mathbf{b}_2}{2}\rceil$. If $\mathbf{b}_1\geq \lceil \frac{\mathbf{b}_1+\mathbf{b}_2}{2}\rceil$, then we want to show that $\Delta_\mathbf{b}$ is a cone over $(0,d,0)$. 
    
    Let $\sigma\in \Delta_\mathbf{b}$, we want to show $\sigma\cup \{(0,d,0)\}\in \Delta_\mathbf{b}$. If $(0,d,0)\in \sigma$, we are done. If $(0,d,0)\notin \sigma$, then it suffices to show $\mathbf{b}_1-\displaystyle\sum_{i\in\sigma}\mathbf{a}^i_1-d\geq 0$. Define $l'\defeq\displaystyle\sum_{i\in \sigma\cup \{(0,d,0)\}}\mathbf{a}^i_0$. Then \[l'=\sum_{i\in \sigma\cup \{(0,d,0)\}}\mathbf{a}^i_0=\sum_{i\in\sigma}\mathbf{a}^i_0\leq \mathbf{b}_0\leq \tilde{l}_j.\]
    Therefore, we have the following 
    \[\mathbf{b}_1\geq \lceil \frac{\mathbf{b}_1+\mathbf{b}_2}{2}\rceil \geq \lceil \frac{dj-\tilde{l}_j}{2}\rceil \geq f_{\tilde{l}_j}\geq f_{l'}\geq \sum_{i\in \sigma\cup \{(0,d,0)\}}\mathbf{a}^i_1=d+\sum_{i\in\sigma}\mathbf{a}^i_1.\]
    Similarly, if $\mathbf{b}_2\geq \lceil \frac{\mathbf{b}_1+\mathbf{b}_2}{2}\rceil$, then $\Delta_\mathbf{b}$ is a cone over $(0,0,d)$.
\end{proof}
Next, we present tables showing the two bounds for cases $d=2,3$, and $4$.

$d=2$
\[
\begin{array}{|c|c|c|c|c|}
\hline
j &3 &4 &5 & j\geq 6 \\ \hline
\tilde{l}_j & 0& 1& 3&3\\ \hline
A_j  & 4&4 & 4&4\\ \hline
\end{array}
\]

$d=3$
    \[
\begin{array}{|c|c|c|c|c|c|c|c|}
\hline
j & 4 & 5 & 6&7 &8 &9 & j\geq 10 \\ \hline
\tilde{l}_j & 0&0 &1 & 3&5 &8 &9 \\ \hline
A_j & 8 & 9&10 & 10&10 &10 & 10\\ \hline
\end{array}
\]

$d=4$
\[
\begin{array}{|c|c|c|c|c|c|c|c|c|c|c|c|}
\hline
j &5 &6 &7 &8 & 9& 10& 11& 12&13 &14 &j\geq15 \\ \hline
\tilde{l}_j &0 & 0& 1&2 & 3& 5& 7&9 & 13& 17& 19\\ \hline
A_j &14 &16 &17 & 18&19 &20 &20 &20 &20 &20 &20 \\ \hline
\end{array}
\]

\section{NON-VANISHING RESULTS}\label{NON-VANISHING RESULTS}
The goal of this section is to prove Theorem~\ref{nonvanishing D}.

\begin{lemma}\label{criticalsimplex}
Let $2\leq p\leq \binom{d+1}{2}-1$, and let $r$ be the unique integer such that $p+1\in I_r\defeq \Big( \binom{r}{2},\, \binom{r+1}{2} \Big]$. Suppose $|\mathbf{b}|=d(p+1)$ and $\mathbf{b}_0=\displaystyle\sum_{i=1}^{p+1}\mathbf{a}^i_0-1$. We define two sets \[D\defeq \Big\{I\in \binom{\{\binom{r}{2}+1,\dots,\binom{r+1}{2}\}}{p+1-\binom{r}{2}}\mid \mathbf{b}_1-\displaystyle\sum_{i=1}^{\binom{r}{2}}\mathbf{a}^i_1-1\leq \displaystyle\sum_{i\in I}\mathbf{a}^i_1 \leq \mathbf{b}_1-\displaystyle\sum_{i=1}^{\binom{r}{2}}\mathbf{a}^i_1 \Big\}\] \[D'\defeq \Big\{\{1,\dots, \binom{r}{2}\}\cup I\mid I\in D \Big\}.\]

Then \[\{\sigma\in\Delta_\mathbf{b}\mid \sigma\cup \{1\}\notin \Delta_\mathbf{b}\}=\{\tau\setminus \{1\}\mid \tau\in D'\}.\]
\end{lemma}
\begin{proof}
    Let $\tau\in D'$. We aim to verify that $\tau\setminus \{1\}\in \Delta_\mathbf{b}$, but $\tau\notin \Delta_\mathbf{b}$.

To show that $\tau\setminus \{1\}\in \Delta_\mathbf{b}$, it suffices to show that $\mathbf{b}-\displaystyle\sum_{i\in\tau\setminus\{1\}}\mathbf{a}^i\in\mathbb{N}^3$. We check this on each component. Note that $\mathbf{b}_0-\displaystyle\sum_{i\in\tau\setminus \{1\}}\mathbf{a}^i_0=\mathbf{b}_0-(\displaystyle\sum_{i\in\tau}\mathbf{a}^i_0-d)=d-1>0$. And
\[\mathbf{b}_2-\displaystyle\sum_{i\in\tau\setminus \{1\}}\mathbf{a}^i_2=\mathbf{b}_2-\displaystyle\sum_{i\in\tau}\mathbf{a}^i_2=\displaystyle\sum_{i\in\tau}\mathbf{a}^i_0-\mathbf{b}_0+\sum_{i\in\tau}\mathbf{a}^i_1-\mathbf{b}_1=\begin{cases}
    0 &\text{if   } \mathbf{b}_1-\displaystyle\sum_{i\in\tau}\mathbf{a}^i_1=1\\
    1 &\text{if   }\mathbf{b}_1-\displaystyle\sum_{i\in\tau}\mathbf{a}^i_1=0. 
\end{cases}\]
To show that $\tau\notin \Delta_\mathbf{b}$, it suffices to show that $\mathbf{b}_0-\displaystyle\sum_{i\in\tau}\mathbf{a}^i_0<0$.

Note that $\mathbf{b}_0-\displaystyle\sum_{i\in\tau}\mathbf{a}^i_0=\mathbf{b}_0-\displaystyle\sum_{i=1}^{p+1}\mathbf{a}^i_0=-1<0$. Thus, $\{\sigma\in\Delta_\mathbf{b}\mid \sigma\cup \{1\}\notin \Delta_\mathbf{b}\}\supseteq\{\tau\setminus \{1\}\mid \tau\in D'\}$.

Conversely, let $\sigma\in\Delta_\mathbf{b}$ but $\sigma\cup\{1\}\notin \Delta_\mathbf{b}$. We want to show that $\sigma\cup\{1\}\in D'$. 

\begin{claim}
The cardinality of $\sigma$ is $p$.
\end{claim}

\begin{proof}[Proof of Claim]
If $|\sigma|=p+1$, then $\mathbf{b}=\displaystyle\sum_{i\in\sigma}\mathbf{a}^i$. Note that $1\notin \sigma$, then we have 
\[
\sum_{i\in\sigma}\mathbf{a}^i_0\leq \sum_{i=1}^{p+2}\mathbf{a}^i_0-d \\
\implies d-1\leq \mathbf{a}^{p+2}_0.
\]
Since $\mathbf{a}^{p+2}_0\leq d-2$, we obtain a contradiction.

If $|\sigma|\leq p-1$, then $\displaystyle\sum_{i\in\sigma}\mathbf{a}^i_0\leq \sum_{i=1}^p\mathbf{a}^i_0-d$. We compute \begin{align*}
    \mathbf{b}_0-d-\sum_{i\in \sigma}\mathbf{a}^i_0 &\geq \displaystyle\sum_{i=1}^{p+1}\mathbf{a}^i_0-1-d-\sum_{i=1}^p\mathbf{a}^i_0+d\\
    &= \mathbf{a}^{p+1}_0-1\\
    &\geq 0.
\end{align*}
This contradicts the assumption that $\sigma \cup \{1\} \notin \Delta_{\mathbf{b}}$.

Therefore, the cardinality of $\sigma$ is $p$, completing the proof of the claim.
\end{proof}
Since $\sigma\in\Delta_\mathbf{b}$ and $\sigma\cup \{1\}\notin \Delta_\mathbf{b}$, we have \(\displaystyle\sum_{i\in\sigma}\mathbf{a}^i_0\geq \mathbf{b}_0+1-d=\displaystyle\sum_{i=2}^{p+1}\mathbf{a}^i_0\). Since $|\sigma|=p$, we have $\displaystyle\sum_{i=2}^{p+1}\mathbf{a}^i_0\geq \displaystyle\sum_{i\in\sigma}\mathbf{a}^i_0$. Thus, $\displaystyle\sum_{i\in\sigma}\mathbf{a}^i_0=\displaystyle\sum_{i=2}^{p+1}\mathbf{a}^i_0$.

\begin{claim}
    $\{2,\cdots,\binom{r}{2}\}\subseteq\sigma$.
\end{claim}
\begin{proof}[Proof of Claim]
If $p=2$, then $r=2$, and the claim holds.

If $p\geq 3$, then $r\geq 3$. Suppose $\{2,\cdots,\binom{r}{2}\}\nsubseteq\sigma$, then there is an $k\in\{2,\cdots,\binom{r}{2}\}$ but $k\notin \sigma$. We have 
\[\sum_{i\in\sigma}\mathbf{a}^i_0\leq \sum_{i=2}^{p+2}\mathbf{a}^i_0-\mathbf{a}^k_0\leq \sum_{i=2}^{p+1}\mathbf{a}^i_0-1.\] This contradicts that $\displaystyle\sum_{i\in\sigma}\mathbf{a}^i_0=\displaystyle\sum_{i=2}^{p+1}\mathbf{a}^i_0$.

\end{proof}
By claim 2, we know that $\displaystyle\sum_{i\in\sigma\setminus\{2,\cdots,\binom{r}{2}\}}\mathbf{a}^i_0=\displaystyle\sum_{i=\binom{r}{2}+1}^{p+1}\mathbf{a}^i_0$. Therefore, $\sigma\setminus\{2,\cdots,\binom{r}{2}\}\in \binom{\{\binom{r}{2}+1,\cdots,\binom{r+1}{2}\}}{p+1-\binom{r}{2}}$.

Finally, we need to show $\mathbf{b}_1-\displaystyle\sum_{i\in\sigma}\mathbf{a}^i_1=1$ or $\mathbf{b}_1-\displaystyle\sum_{i\in\sigma}\mathbf{a}^i_1=0$. 

Assume not, then $\mathbf{b}_1-\displaystyle\sum_{i\in\sigma}\mathbf{a}^i_1\geq 2$. Note that \begin{align*}
    \mathbf{b}_0-\displaystyle\sum_{i\in\sigma}\mathbf{a}^i_0&=\mathbf{b}_0-\displaystyle\sum_{i=2}^{p+1}\mathbf{a}^i_0\\
    &=\sum_{i=1}^{p+1}\mathbf{a}^i_0-1-\sum_{i=2}^{p+1}\mathbf{a}^i_0\\
    &=d-1.
\end{align*}
We compute
\begin{align*}
     \mathbf{b}_2-\displaystyle\sum_{i\in\sigma}\mathbf{a}^i_2&=d(p+1)-\mathbf{b}_0-\mathbf{b}_1-(dp-\displaystyle\sum_{i\in\sigma}\mathbf{a}^i_0-\displaystyle\sum_{i\in\sigma}\mathbf{a}^i_1)\\
     &=d-(\mathbf{b}_0-\displaystyle\sum_{i\in\sigma}\mathbf{a}^i_0)-(\mathbf{b}_1-\displaystyle\sum_{i\in\sigma}\mathbf{a}^i_1)\\
     &\leq d-(d-1)-2\\
     &=-1,
\end{align*}
which is a contradiction.

Therefore, $\sigma\cup\{1\}\in D'$. This means that $\{\sigma\in\Delta_\mathbf{b}\mid \sigma\cup \{1\}\notin \Delta_\mathbf{b}\}\subseteq\{\tau\setminus \{1\}\mid \tau\in D'\}$. The lemma is proved.
\end{proof}

\begin{theorem}\label{wedgesum}
   Under the same assumptions as in Lemma~\ref{criticalsimplex}, the simplicial complex $\Delta_{\mathbf{b}}$ is homotopy equivalent to the wedge sum of $\#D$ copies of the sphere $S^{p-1}$. In particular, the Betti number $\beta_{p,\mathbf{b}}$ equals $\#D$.

\end{theorem}
\begin{proof}
  By Lemma~\ref{criticalsimplex}, the cardinality of $\{\sigma\in\Delta_\mathbf{b}\mid \sigma\cup \{1\}\notin \Delta_\mathbf{b}\}$ is equal to $\#D$ and \[N_{1,q}=\begin{cases}
      \#D &\text{if } q=p-1,\\
      0&\text{if }q\neq p-1.
  \end{cases}\] By Corollary~\ref{Nvq}, the simplicial complex $\Delta_\mathbf{b}$ is homotopy equivalent to a CW complex with exactly one $0$-cell and $\#D$  cells of dimension $(p-1)$. Hence, it is homotopy equivalent to the wedge sum of $\#D$ copies of $S^{p-1}$. By Hochster’s formula, the Betti number $\beta_{p,\mathbf{b}}$ equals $\#D$.
    
\end{proof}
\begin{corollary}
    Under the same assumptions as in Lemma~\ref{criticalsimplex}, then $\beta_{p,\mathbf{b}}\neq 0$ if and only if  $\displaystyle\sum_{i=1}^{p+1}\mathbf{a}^i_2\leq \mathbf{b}_1\leq 1+\displaystyle\sum_{i=1}^{p+1}\mathbf{a}^i_1$.
\end{corollary}
\begin{proof}
By Theorem \ref{wedgesum}, the Betti number  $\beta_{p,\mathbf{b}}\neq 0$ if and only if $\mathbf{b}_1-\displaystyle\sum_{i=1}^{\binom{r}{2}}\mathbf{a}^i_1-1\leq \displaystyle\sum_{i\in I}\mathbf{a}^i_1 \leq \mathbf{b}_1-\displaystyle\sum_{i=1}^{\binom{r}{2}}\mathbf{a}^i_1 $ for some $I\in \binom{\{\binom{r}{2}+1,\cdots,\binom{r+1}{2}\}}{p+1-\binom{r}{2}}$. The latter condition is equivalent to 
    $\displaystyle\sum_{i=1}^{p+1}\mathbf{a}^i_2\leq \mathbf{b}_1\leq 1+\displaystyle\sum_{i=1}^{p+1}\mathbf{a}^i_1$.
\end{proof}

It is well known that for the projective space case ($m=2$), $\beta_{p,1}\neq 0$ if and only if $1\leq p\leq \binom{d+1}{2}$. Understanding the multigraded Betti numbers when $p=\binom{d+1}{2}$ is interesting, and we study this case in the next theorem.
\begin{theorem}
    Let $p=\binom{d+1}{2}$. Suppose $|\mathbf{b}|=d(p+1)$ and $\mathbf{b}_0=\displaystyle\sum_{i=1}^{p+1}\mathbf{a}^i_0-1$. Then \[\displaystyle\sum_{i=1}^p\mathbf{a}^i_1+1\leq \mathbf{b}_1\leq\displaystyle\sum_{i=1}^{p+1}\mathbf{a}^i_1 \text{ if and only if }\beta_{p,\mathbf{b}}\neq 0.\] 
    
    In particular, $\beta_{p,\mathbf{b}}=1$ for $\mathbf{b}_1$ in this range.
\end{theorem}
\begin{proof}
   If $\displaystyle\sum_{i=1}^p\mathbf{a}^i_1+1\leq \mathbf{b}_1\leq\displaystyle\sum_{i=1}^{p+1}\mathbf{a}^i_1$, then $\{\sigma\in\Delta_\mathbf{b}\mid \sigma\cup \{1\}\notin \Delta_\mathbf{b}\}$ equals
   \[\Big\{\{2,\dots,p\}, \{2,\dots,p,p+1+d-\mathbf{b}_1+\displaystyle\sum_{i=1}^p\mathbf{a}^i_1\},\{2,\dots,p,p+2+d-\mathbf{b}_1+\displaystyle\sum_{i=1}^p\mathbf{a}^i_1\}\Big\}.\]

We can add the pair $\{2,\dots,p\}<\{2,\dots,p,p+1+d-\mathbf{b}_1+\displaystyle\sum_{i=1}^p\mathbf{a}^i_1\}$ to the collection $V=\{\sigma<\sigma\cup \{1\}\}$ defined in Example \ref{dvf}. This new collection is a gradient vector field of a discrete Morse function. To see this, suppose there is a non-trivial closed $V$-path 
\[\{2,\dots,p\}, \{2,\dots,p,p+1+d-\mathbf{b}_1+\displaystyle\sum_{i=1}^p\mathbf{a}^i_1\},\sigma_1, \sigma_1\cup\{1\},\cdots,\sigma_h, \sigma_h\cup\{1\},\sigma_{h+1} \ \text{for some $h\geq 1$,}\] then $1\in\sigma_2$ and $h=1$, which contradicts $\sigma_2=\{2,\dots,p\}$. 

   The critical simplices are $\{1\}$ and $\{2,\dots,p,p+2+d-\mathbf{b}_1+\displaystyle\sum_{i=1}^p\mathbf{a}^i_1\}$. This means that $\Delta_\mathbf{b}$ is homotopy to a $(p-1)$-sphere, and hence $\beta_{p,\mathbf{b}}=1$.

    If $\mathbf{b}_1=\displaystyle\sum_{i=1}^p\mathbf{a}^i_1$ or $\mathbf{b}_1=\displaystyle\sum_{i=1}^{p+1}\mathbf{a}^i_1+1$, then $\{\sigma\in\Delta_\mathbf{b}\mid \sigma\cup \{1\}\notin \Delta_\mathbf{b}\}$ equals 
    \[\Big\{\{2,\dots,p\}, \{2,\dots,p,p+1+d-\mathbf{b}_1+\displaystyle\sum_{i=1}^p\mathbf{a}^i_1\}\Big\}\] or 
    \[\Big\{\{2,\dots,p\}, \{2,\dots,p,p+2+d-\mathbf{b}_1+\displaystyle\sum_{i=1}^p\mathbf{a}^i_1\}\Big\} \ \text{respectively}.\]
In each of these two cases, We add this pair to $V$. For the same reason as above, the new collections is a gradient vector field of a discrete Morse function. The only criticial simplex is $\{1\}$. This means that $\Delta_\mathbf{b}$ is contractible, and hence $\beta_{p,\mathbf{b}}=0$.

    If $\mathbf{b}_1<\displaystyle\sum_{i=1}^p\mathbf{a}^i_1$ or $\mathbf{b}_1>\displaystyle\sum_{i=1}^{p+1}\mathbf{a}^i_1+1$, then $\{\sigma\in\Delta_\mathbf{b}\mid \sigma\cup \{1\}\notin \Delta_\mathbf{b}\}$ is empty.
In this case, the only criticial simplex is $\{1\}$. This means that $\Delta_\mathbf{b}$ is contractible, and hence $\beta_{p,\mathbf{b}}=0$.
\end{proof}
\section{GENERALIZATION TO HIGHER DIMENSIONAL PROJECTIVE SPACE}\label{general}
In Sections~\ref{VANISHING THEOREM 1}, \ref{VANISHING THEOREM 2}, and~\ref{NON-VANISHING RESULTS}, we proved Theorems~\ref{vanishing1}, \ref{vanishing2}, and~\ref{nonvanishing D} for the case $m=2$. The same arguments apply to $m>2$, and we state the general results in this section. Throughout this section, we assume $m\geq 2$.

\begin{theorem}
      Let $j$ be a positive integer, and let $|\mathbf{b}|=dj$. If $\mathbf{b}_0\geq A_j\defeq \displaystyle\sum_{i\in\{1,...,j\}\cap \{1,...,n\}}\mathbf{a}^i_0$, then $\Delta_\mathbf{b}$ is a cone over $(d,0,\ldots,0)$, where $(d,0,\ldots,0)$ lies in $\mathbb{N}^{m+1}$.
\end{theorem}
\begin{corollary}
      Let $j$ be a positive integer, and let $|\mathbf{b}|=dj$. If $\mathbf{b}_0\geq A_j$, then $\beta_{p,\mathbf{b}}=0$ for all $p\in\mathbb{Z}$.
\end{corollary}

We use the same definition of $f_l$ as before, and $f_l$ is again an increasing function. Assume $j\geq {d+m-1\choose m-1}$. The function $\phi(l)\defeq \lceil \frac{dj-l}{m} \rceil-f_l$ is a decreasing function on $\{0,...,A_j-1\}$. Let $\tilde{l}_j$ be the largest integer $l$ in $\{0,...,A_j-1\}$ such that $\phi(l)\geq 0$. 
 \begin{theorem}
   Let $j$ be a positive integer such that $j\geq {d+m-1\choose m-1}$, and let $|\mathbf{b}|=dj$. If $\mathbf{b}_0\leq \tilde{l}_j$, then $\Delta_\mathbf{b}$ is a cone over one of the points $(0,d,0,\ldots,0)$, $(0,0,d,0,\ldots,0)$, $\ldots$ , $(0,\ldots,0,d)$ in $\mathbb{N}^{m+1}$. 
 \end{theorem}
 
 \begin{corollary}
       Let $j$ be a positive integer such that $j\geq {d+m-1\choose m-1}$, and let $|\mathbf{b}|=dj$. If $\mathbf{b}_0\leq \tilde{l}_j$, then $\beta_{p,\mathbf{b}}=0$ for all $p\in \mathbb{Z}$.
 \end{corollary}

\begin{lemma}\label{m-criticalsimplex}
Let $m\leq p\leq \binom{d+m-1}{m}-1$, and let $r$ be the unique integer such that $p+1\in I_r\defeq \Big(\binom{r+m-2}{m},\binom{r+m-1}{m}\Big]$. Suppose $|\mathbf{b}|=d(p+1)$ and $\mathbf{b}_0=\displaystyle\sum_{i=1}^{p+1}\mathbf{a}^i_0-1$. We define two sets 
\begin{gather*}
    D\defeq \Bigg\{I\in \binom{\{\binom{r+m-2}{m}+1,\dots,\binom{r+m-1}{m}\}}{p+1-\binom{r+m-2}{m}}\mid \hspace*{.5\linewidth}
    \\
    \hspace*{.3\linewidth}
    \mathbf{b}_t-\displaystyle\sum_{i=1}^{\binom{r+m-2}{m}}\mathbf{a}^i_t-1\leq \displaystyle\sum_{i\in I}\mathbf{a}^i_t \leq \mathbf{b}_t-\displaystyle\sum_{i=1}^{\binom{r+m-2}{m}}\mathbf{a}^i_t \quad \forall 1\leq t\leq m-1\Bigg\}
\end{gather*}
and   
    \[D'\defeq \Big\{\{1,\dots,\binom{r+m-2}{m}\}\cup I\mid I\in D\Big\}.\]
Then \[\{\sigma\in\Delta_\mathbf{b}\mid \sigma\cup \{1\}\notin \Delta_\mathbf{b}\}=\{\tau\setminus \{1\}\mid \tau\in D'\}.\]


\end{lemma}
\begin{theorem}\label{mbd}
   Under the same assumptions as in Lemma~\ref{m-criticalsimplex}, the simplicial complex $\Delta_{\mathbf{b}}$ is homotopy equivalent to the wedge sum of $\#D$ copies of the sphere $S^{p-1}$. In particular, the Betti number $\beta_{p,\mathbf{b}}$ equals $\#D$.

\end{theorem}
\section{Optimality of the Bounds $A_{p+1}$}\label{optimal}
The articles~\cite{zbMATH06117694, zbMATH06609392} show that $\beta_{p,1}\neq 0$ for $1\leq p\leq \binom{d+m-1}{m}+m-2$. In the following, we give a combinatorial proof and show that, on the row $1$ of the Betti table, the bounds $A_{p+1}$ for the multigraded Betti numbers are sharp.
\begin{theorem}\label{optimalA}
   Let $d\geq 2$, $m\geq 2$. For each $1\leq p\leq \binom{d+m-1}{m}+m-2$, there is a $\mathbf{b}\in \mathbb{N}^{m+1}$ with $|\mathbf{b}|=d(p+1)$ and $\mathbf{b}_0=\displaystyle\sum_{i=1}^{p+1}\mathbf{a}^i_0-1$ such that $\beta_{p,\mathbf{b}}\neq 0$. 
\end{theorem}
\begin{proof}
    We define $\mathbf{b}$ in each coordinate as follows: $\mathbf{b}_0\defeq \displaystyle\sum_{i=1}^{p+1}\mathbf{a}^i_0-1$, $\mathbf{b}_s\defeq\displaystyle\sum_{i=1}^{p+1}\mathbf{a}^i_s$ for each $1\leq s\leq m-1$, $\mathbf{b}_m\defeq\displaystyle\sum_{i=1}^{p+1}\mathbf{a}^i_m+1$.
\begin{itemize}
    \item $\binom{d+m-1}{m}\leq p \leq \binom{d+m-1}{m}+m-2$: 
Let \(F\defeq\{\sigma\in\Delta_\mathbf{b}\mid \{2,\dots,\binom{d+m-1}{m}\}\subseteq \sigma\}\).  

    It is straightforward to verify that \[\{\sigma\in \Delta_{\mathbf{b}}\mid \sigma\cup \{1\}\notin \Delta_{\mathbf{b}}\}=F.\]
    It is also straightforward to verify that 
    \begin{align*}
       & \Big\{\tau\in F\mid \tau\cup \{\binom{d+m-1}{m}+1\}\notin F\Big\}\\
        =&\Big\{(\{2,\dots,p+1\}\setminus \{\binom{d+m-1}{m}+1\})\cup \{\binom{d+m-1}{m}+m\}\Big\}.
    \end{align*}
Let $V$ be the collection of pairs \(
\{\sigma < \sigma \cup \{1\} \}
\), where $\sigma$ is a nonempty simplex, $1 \notin \sigma$, and $\sigma \cup \{1\}$ is a simplex. Let $W$ be the collection of pairs $\Big\{\tau <\tau\cup \{\binom{d+m-1}{m}+1\} \Big\}$, where $\tau\in F$, $\binom{d+m-1}{m}+1\notin \tau$, and $\tau\cup \{\binom{d+m-1}{m}+1\}\in F$.

 Then $V\cup W$ is a gradient vector field of a discrete Morse function, and the critical simplices are $\{1\}$ and $(\{2,\dots,p+1\}\setminus \{\binom{d+m-1}{m}+1\})\cup \{\binom{d+m-1}{m}+m\}$. This means $\Delta_{\mathbf{b}}$ is homotopy to a $(p-1)$-sphere, and hence $\beta_{p,\mathbf{b}}=1$. 

 \item $m\leq p\leq \binom{d+m-1}{m}-1 $: Let $r$ be the unique integer such that $p+1\in I_r\defeq (\binom{r-2+m}{m},\binom{r-1+m}{m}]$. Note that $\{\binom{r-2+m}{m}+1,\dots,p+1\}\in D$, we know $\beta_{p,\mathbf{b}}\geq 1$ by Theorem \ref{mbd}.  
 
 \item $1\leq p\leq m-1$:  It is straightforward to verify that \[\{\sigma\in \Delta_{\mathbf{b}}\mid \sigma\cup \{1\}\notin \Delta_{\mathbf{b}}\}=\Big\{\sigma\subseteq \{2,\dots,p+1\}\cup\{m+1\}\mid |\sigma|\geq p\Big\}.\]
   Let $V$ be the collection of pairs \(
\{\sigma < \sigma \cup \{1\} \}
\), where $\sigma$ is a nonempty simplex, $1 \notin \sigma$, and $\sigma \cup \{1\}$ is a simplex. 

Then $V\cup \Big\{\{3,\dots,p+1\}\cup\{m+1\}<\{2,\dots,p+1\}\cup\{m+1\}\Big\}$ is a gradient vector field of a discrete Morse function, and the critical simplices are $\{1\}$ and $\{\sigma\subseteq \{2,\dots,p+1\}\cup\{m+1\}\mid |\sigma|= p  \text{ and }2\in\sigma\}$. 

Since $|\{\sigma\subseteq \{2,\dots,p+1\}\cup\{m+1\}\mid |\sigma|= p  \text{ and }2\in\sigma\}|=p$, this means $\Delta_{\mathbf{b}}$ is homotopy to the wedge sum of $p$ copies of the sphere $S^{p-1}$. Hence, $\beta_{p,\mathbf{b}}=p$.

\end{itemize}    
\end{proof}

 

\bibliographystyle{alpha} 
\bibliography{references}

\bigskip

\noindent
Christian Haase: Mathematik, Freie Universität Berlin; \url{haase@math.fu-berlin.de}
\medskip

\noindent
Zongpu Zhang: Mathematik, Freie Universität Berlin; \url{zongpuzhang@zedat.fu-berlin.de}

\end{document}